\documentclass [11pt,twoside,a4paper]{article}
\usepackage{amsfonts}
\usepackage{amsthm}
\usepackage{amsmath}
\usepackage{amstext}
\usepackage{amssymb}
\usepackage{mathrsfs}
\usepackage{amscd}
\usepackage{xypic}
\usepackage{epsf}              %ͼÐκê°ü
\usepackage{graphicx}          %ÕâÊÇͼÐÎÃüÁîµÄºê°ü
\usepackage{fancybox}          %ÕâÊǺÐ×ÓÃüÁîµÄºê°ü
\usepackage{color}             %ÕâÊÇÑÕÉ«ºê°ü
\usepackage{fancyhdr}
\usepackage[hang,footnotesize]{caption2}  %ÓйØͼÐαêÌâµÄºê°ü

\def\mathcal{\mathscr}
\newfont{\aaa}{cmb10 at 19pt}
\newfont{\bbb}{cmb10 at 11pt}
\newtheorem{lem}{Lemma}[]
\newtheorem{thm}{Theorem}
\newtheorem{cor}{Corollary}[]

\newtheorem{definition}{Definition}[]

\def\Ext{\text{Ext}}

\def\End{\text{End}}
\def\ker{\text{Ker}}
\def\cok{\text{Coker}}

\def\dim{\text{dim}}
\def\rad{\text{rad}}

\def\modl{\text{mod}}

\def\ind{\text{ind}}

\def\soc{\text{Soc}}
\def\mfk{\mathfrak}

\def\ZZ{\mathbb Z}

\newcommand{\beq}{\begin{equation}}
\newcommand{\eeq}{\end{equation}}
\newcommand{\bey}{\begin{eqnarray}}
\newcommand{\eey}{\end{eqnarray}}
\newcommand{\beyy}{\begin{eqnarray*}}
\newcommand{\eeyy}{\end{eqnarray*}}

 \makeatletter
\def\@evenhead{
   \vbox{\hbox to \textwidth
{}{\hspace{0mm}{\footnotesize \thepage}}{\hspace{8cm}
   {\footnotesize {Ming Ding and Jie Sheng}}}
  \protect\vspace{1truemm}\relax
   \hrule depth0pt height0.15truemm width\textwidth
   }}
   \def\@evenfoot{}
\def\@oddhead{
    \vbox{\hbox to \textwidth
  {{\hspace{0cm}{\footnotesize Multiplicative properties of  a quantum
Caldero-Chapoton map}
  \hfill{\footnotesize \thepage}}\hspace{0mm}}{}
   \protect\vspace{1truemm}\relax
   \hrule depth0pt height0.15truemm width\textwidth
  }}
  \def\@oddfoot{}
\makeatother
\begin{document}

%Ê×ҳüÉ趨%%%%%%%%%%%%%%%%%%%%%%%%%%%%%%%%%%%%%%%%%%%%%%%%%%
%\setcounter{page}{61} %
%\thispagestyle{empty} \thispagestyle{fancy} {
  %\fancyhead[lO,RE]{\footnotesize  Front. Math. China \\
 % DOI .../...-...-...-...\\[3mm]
  %\includegraphics[0,-50][0,0]{11.bmp}}
  %\fancyhead[RO,LE]{\scriptsize \bf %http://www.spcum.hep.com.cn\\
%} \fancyfoot[CE,CO]{}}
%\renewcommand{\headrulewidth}{0pt}%(Ê×Ò³ÊéüÏß¿í£©
%\renewcommand{\headsep}{0.7cm}%£¨Ê×Ò³ÕýÎÄÓëÊéü¾àÀ룩

%Ê×ҳüÉ趨%%%%%%%%%%%%%%%%%%%%%%%%%%%%%%%%%%%%%%%%%%%%%%%%%%

%%%ÒÔÏÂÕýÎÄ¿ªÊ¼
\setcounter{page}{1}
\qquad\\[8mm]

\noindent{\aaa{Multiplicative properties of a  quantum
Caldero-Chapoton map associated to valued quivers}}\\[1mm]

\noindent{\bbb Ming Ding$^{1}$ and Jie Sheng$^{2}$}\\[-1mm]

 \noindent\footnotesize{1. School of Mathematical Sciences,
Nankai University, Tianjin 300071, China\\
E-mail:m-ding04@mails.tsinghua.edu.cn}\\[4mm]
\noindent\footnotesize{2. Department of Applied Mathematics, China
Agricultural University, Beijing 100083, China\\
E-mail:
shengjie@amss.ac.cn}\\[6mm]
%\vskip-2mm \noindent{\footnotesize$\copyright$ Higher Education
%Press and Springer-Verlag Berlin Heidelberg 2011} \vskip 4mm

\normalsize\noindent{\bbb Abstract}\quad We prove a multiplication
theorem of  a quantum Caldero-Chapoton map associated to valued
quivers which extends the results in \cite{DX}\cite{D}. As an
application, when $Q$ is a valued quiver of finite type or rank $2$,
we obtain that the algebra $\mathcal{AH}_{|k|}(Q)$ generated by all
cluster characters (see Definition \ref{def}) is exactly the quantum
cluster algebra $\mathcal{EH}_{|k|}(Q)$ and various bases of the
quantum cluster algebras of rank $2$ can  naturally be deduced.
\vspace{0.3cm}

\footnotetext{Corresponding author: Jie Sheng, E-mail:
shengjie@amss.ac.cn}

\noindent{\bbb Keywords}\quad cluster variable, quantum cluster algebra\\
{\bbb MSC}\quad 16G20\\[0.4cm]

\noindent{\bbb{1\quad Introduction}}\\[0.1cm]
Ever since the emergency of cluster algebras, the close relation
between it and quiver representations has always been emphasized.
One interesting viewpoint is to consider cluster algebras as some
kind of Hall algebras of quiver representations, that was
particularly enhanced by \cite{caldchap} in which the so-called
Caldero-Chapoton formula (or map, or character) was invented. Then
the multiplication formulas (\cite{CK2005},\cite{DX0},\cite{Hubery})
of Caldero-Chapoton characters become important, especially in the
construction of integral bases of cluster algebras (e.g. see
\cite{DX,DXX}).

In \cite{rupel}, D.~Rupel obtained a quantum analogue of the
Caldero-Chapoton formula, which is crucial for the study of quantum
cluster algebras. Unlike in the cluster algebras, it does not
generally hold that $X_{N}X_{M}=|k|^{\pm\frac{1}{2}d_{N\oplus
M}}X_{N\oplus M}$ for any $d_{N\oplus M}\in \ZZ$.  A natural
question to ask is whether the quantized Caldero-Chapoton formula
could be extended to the cluster category. In \cite{D}, this aim was
achieved for equally-valued quivers and a multiplication formula was
verified, which implies that for finite type the algebra
$\mathcal{AH}_{|k|}(Q)$ generated by all cluster characters (see
Definition \ref{def}) is exactly the quantum cluster algebra
$\mathcal{EH}_{|k|}(Q)$.

In this paper, we will extend the results in \cite{D} to valued
quivers and prove two multiplication formulas therein. As an
application, when $Q$ is a valued quiver of finite type or rank $2$,
we prove that the algebra $\mathcal{AH}_{|k|}(Q)$ is exactly the
quantum cluster algebra $\mathcal{EH}_{|k|}(Q)$. In particular, we
obtain various bases of the quantum cluster algebras of rank $2$ by
using the standard monomials in \cite{berzel}.
\noindent \\[4mm]

\noindent{\bbb 2\quad Preliminaries and statement of the main result}\\[0.1cm]
\noindent{\bbb 2.1\quad Definition of quantum cluster algebras}
 Let $L$ be a lattice of rank $m$ and $\Lambda:L\times L\to
\ZZ$ a skew-symmetric bilinear form. Note that $\Lambda$ can  be
identified with an $m\times m$ skew-symmetric matrix which still
denoted by $\Lambda$ if there is no confusion. Set a formal variable
$q$ and
 the ring of integer Laurent polynomials $\ZZ[q^{\pm1/2}]$.
Define the \textit{based quantum torus} associated to the pair
$(L,\Lambda)$ to be the $\ZZ[q^{\pm1/2}]$-algebra $\mathcal{T}$ with
a distinguished $\ZZ[q^{\pm1/2}]$-basis $\{X^e: e\in L\}$ and the
multiplication
\[X^eX^f=q^{\Lambda(e,f)/2}X^{e+f}.\]
  It is known that $\mathcal{T}$ is  contained in its
skew-field of fractions $\mathcal{F}$. A \textit{toric frame} in
$\mathcal{F}$ is a map $M: \ZZ^m\to \mathcal{F} \setminus \{0\}$
given by
\[M({\bf c})=\varphi(X^{\eta({\bf c})})\] where $\varphi$ is an
automorphism of $\mathcal{F}$ and $\eta: \ZZ^m\to L$ is an
isomorphism of lattices. By the definition, the elements $M({\bf
c})$ form a $\ZZ[q^{\pm1/2}]$-basis of the based quantum torus
$\mathcal{T}_M:=\varphi(\mathcal{T})$ and satisfy the following
relations:
\[M({\bf c})M({\bf d})=q^{\Lambda_M({\bf c},{\bf d})/2}M({\bf c}+{\bf d}),\
M({\bf c})M({\bf d})=q^{\Lambda_M({\bf c},{\bf d})}M({\bf d})M({\bf
c}),\]
\[ M({\bf 0})=1,\ M({\bf c})^{-1}=M(-{\bf c}),\]
where $\Lambda_M$ is the skew-symmetric bilinear form on $\ZZ^m$
obtained from the lattice isomorphism $\eta$.  Let $\Lambda_M$ be
 the skew-symmetric $m\times m$ matrix defined by
$\lambda_{ij}=\Lambda_M(e_i,e_j)$ where $\{e_1, \ldots, e_m\}$ is
the standard basis of $\ZZ^m$.  Given a toric frame $M$, let
$X_i=M(e_i)$.  Then we have
$$\mathcal{T}_M=\ZZ[q^{\pm1/2}]\langle X_1^{\pm 1}, \ldots,
X_m^{\pm1}:X_iX_j=q^{\lambda_{ij}}X_jX_i\rangle.$$  An easy
computation shows that:
\[M({\bf c})=q^{\frac{1}{2}\sum_{i<j}
c_ic_j\lambda_{ji}}X_1^{c_1}X_2^{c_2}\cdots X_m^{c_m}=:X^{({\bf c})}
\ \ \ ({\bf c}\in\ZZ^m).\]

Let $\Lambda$ be an $m\times m$ skew-symmetric matrix and
$\tilde{B}$  an $m\times n$ matrix with $n\le m$.  We call the pair
$(\Lambda, \tilde{B})$ \textit{compatible} if up to permuting rows
and columns $\tilde{B}^T\Lambda=(D|0)$  with
$D=diag(d_1,\cdots,d_n)$ where $d_i\in \mathbb{N}$ for $1\leq i\leq
n$. The pair $(M,\tilde{B})$ is called a \textit{quantum seed} if
the pair $(\Lambda_M, \tilde{B})$ is compatible.  Define the
$m\times m$ matrix $E=(e_{ij})$ as follows
\[e_{ij}=\begin{cases}
\delta_{ij} & \text{if $j\ne k$;}\\
-1 & \text{if $i=j=k$;}\\
max(0,-b_{ik}) & \text{if $i\ne j = k$.}
\end{cases}
\]
For $n,k\in\ZZ$, $k\ge0$, denote ${n\brack
k}_q=\frac{(q^n-q^{-n})\cdots(q^{n-k+1}-q^{-n+k-1})}{(q^k-q^{-k})\cdots(q-q^{-1})}$.
Let $k\in[1,n]$ where $[1,n]=\{1,\cdots,n\}$ and ${\bf
c}=(c_1,\ldots,c_m)\in\ZZ^m$ with $c_{k}\geq 0$. Define the toric
frame $M': \ZZ^m\to \mathcal{F} \setminus \{0\}$ as follows
\begin{equation}\label{eq:cl_exp}M'({\bf c})=\sum^{c_k}_{p=0} {c_k \brack p}_{q^{d_k/2}} M(E{\bf c}+p{\bf b}^k),\ \ M'({\bf -c})=M'({\bf c})^{-1}.\end{equation}
where the vector ${\bf b}^k\in\ZZ^m$ is the $k$th column of
$\tilde{B}$.  Following \cite{ca1}, we say a real $m\times n$ matrix
$\tilde{B}'$ is obtained from $\tilde{B}$  by matrix mutation in
direction $k$ if the entries of $\tilde{B}'$ are given by
\[b'_{ij}=\begin{cases}
-b_{ij} & \text{if $i=k$ or $j=k$;}\\
b_{ij}+\frac{|b_{ik}|b_{kj}+b_{ik}|b_{kj}|}{2} & \text{otherwise.}
\end{cases}
\]
Then the quantum seed $(M',\tilde{B}')$ is defined to be the
mutation of $(M,\tilde{B})$ in direction $k$. Two quantum seeds are
called mutation-equivalent if they can be obtained from each other
by a sequence of mutations. Let $\mathcal{C}=\{M'(e_i): i\in[1,n]\}$
where $(M',\tilde{B}')$ is mutation-equivalent to $(M,\tilde{B})$.
The elements of $\mathcal{C}$ are called the \textit{cluster
variables}. Let $\mathbb{P}=\{M(e_i): i\in[n+1,m]\}$ and  the
elements of $\mathbb{P}$ are called \emph{coefficients}. Denote by
$\ZZ\mathbb{P}$ the ring of Laurent polynomials generated by
$q^{\frac{1}{2}},\mathbb{P}$ and their inverses. Then the
\textit{quantum cluster algebra}
$\mathcal{A}_q(\Lambda_M,\tilde{B})$ is defined to be the
$\ZZ\mathbb{P}$-subalgebra of $\mathcal{F}$ generated by
$\mathcal{C}$.\\[0.1cm]

\noindent{\bbb 2.2\quad The quantum Caldero-Chapoton map and  main
result} Let $k$ be a finite field with cardinality $|k|=q$ and
$m\geq n$ be two positive integers. Let $\Delta$ be a valued graph
without vertex loops and with vertex set $\{1,\ldots,m\}$. The edges
of $\Delta$ are of the form $\xymatrix{i\ar@{-}[r]^{(a_{ij},a_{ji})}
& j}$, in which the positive integers $a_{ij}$ form a symmetrizable
matrix.

Let $\widetilde{Q}$ be an orientation of $\Delta$ containing no
oriented cycles: that is, we replace each valued edge by a valued
arrow. Thus $\widetilde{Q}$ is called a valued quiver. Note that any
finite dimensional basic hereditary $k$-algebra can be obtained by
taking the tensor algebra of the $k$-species associated to
$\widetilde{Q}$. In what follows we will denote by
$\widetilde{\mathfrak{S}}$ the
 $k$-species of type $\widetilde{Q}$ in the sense of
\cite{Hubery}, which identified a $k$-species with its corresponding
tensor algebra.

The full subquiver $Q$ on the vertices $1,\ldots,n$ is called the
principal part of $\widetilde{Q}$, with the corresponding
$k$-species denoted by $\mathfrak{S}$. For $1\leq i\leq m$, let
$S_i$ be the $i$-th simple module for $\widetilde{\mathfrak{S}}$.

Let $\widetilde{B}$ be the $m\times n$ matrix associated to the
quiver $\widetilde{Q}$ whose entry in position $(i,j)$ given by
$$
b_{ij}=\dim_{\End_{\widetilde{\mathfrak{S}}}(S_{i})^{op}}\Ext^{1}_{\widetilde{\mathfrak{S}}}(S_{i},S_{j})-
\dim_{\End_{\widetilde{\mathfrak{S}}}(S_{i})}\Ext^{1}_{\widetilde{\mathfrak{S}}}(S_{j},S_{i})
$$
for $1\leq i\leq m$, $1\leq j\leq n$. Denote by $\widetilde{I}$ the
left $m\times n$ submatrix of the identity matrix of size $m\times
m$.

By \cite{rupel}, we can assume that there exists some antisymmetric
$m\times m$ integer matrix $\Lambda$ such that
\begin{align*}
\Lambda(-\widetilde{B})=\begin{bmatrix}D_n\\0
\end{bmatrix},
\end{align*}
where $D_n=diag(d_1,\cdots,d_n)$ where $d_i\in \mathbb{N}$ for
$1\leq i\leq n$. Let $\widetilde{R}=\widetilde{R}_{\widetilde{Q}}$
be the $m\times n$ matrix with its entry in position $(i,j)$ is
\[
\widetilde{r}_{ij}:=\mathrm{dim}_{\End_{\widetilde{\mathfrak{S}}}(S_{i})}\Ext^{1}_{\widetilde{\mathfrak{S}}}(S_j,S_i)
\] for $1\leq i\leq m$, $1\leq j\leq n$ respectively. And define
$\widetilde{R}^{'}:=\widetilde{R}_{\widetilde{Q}^{op}}.$
 Denote the principal parts of
the matrices $\widetilde{B}$ and $\widetilde{R}$ by $B$ and $R$
respectively. Note that
$\widetilde{B}=\widetilde{R}^{'}-\widetilde{R}$ and $B=R^{'}-R$.

Let $\mathcal C_{Q}$ be the cluster category (see \cite{BMRRT}) of
the valued quiver $Q$, i.e., the orbit category of the derived
category $\mathcal{D}^b(\mathfrak{S})$ by the functor
$F=\tau\circ[-1]$. We note that the indecomposable
$\mathfrak{S}$-modules and $P_i[1]$ for $1\leq i \leq n$ exhaust the
indecomposable objects of the cluster category $\mathcal C_Q$:
$$\ind\ \mathcal C_Q=\ind\ \modl\mathfrak{S}\sqcup\{P_i[1]:1\leq i \leq n\}$$ where $P_i$ is
the indecomposable projective $\mathfrak{S}$-module at $i$ for $i=1,
\cdots, n.$ Each object $M$ in $\mathcal C_Q$ can be uniquely
decomposed in the following way:
$$M=M_0\oplus P_M[1]$$
where $M_0$ is a $\mathfrak{S}$-module and $P_M$ is a projective
module. Let $P_M=\bigoplus_{1\leq i \leq n}m_iP_i.$ We extend the
definition of the dimension vector $\mathrm{\underline{dim}}$ on
modules in $\mathrm{mod}\mathfrak{S}$ to objects in $\mathcal C_Q$
by setting
$$\mathrm{\underline{dim}}M=\mathrm{\underline{dim}}M_0-(m_i)_{1\leq i \leq n}.$$
The Euler form on $\mathfrak{S}$-modules $M$ and $N$ is given by
$$\langle M,N\rangle=\mathrm{dim}_{k}\mathrm{Hom}_{\mathfrak{S}}(M,N)-\mathrm{dim}_{k}\mathrm{Ext}^{1}_{\mathfrak{S}}(M,N).$$
Note that the Euler form only depends on the dimension vectors of
$M$ and $N$ and the matrix representing this form is
$(I_{n}-R^{tr})D_n=D_n(I_{n}-R^{'})$.

The quantum Caldero-Chapoton map of an acyclic quiver $Q$ has been
defined in \cite{rupel}\cite{fanqin}\cite{D}. The quantum
Caldero-Chapoton map was defined in \cite{rupel} for
$\mathfrak{S}$-modules, in \cite{fanqin} for coefficient-free rigid
object in $\mathcal C_{\widetilde{Q}}$. Later it was extended in
\cite{D} to the cluster category for equally-valued quivers. For
valued quivers, we also have
$$X_?: \mathrm{Obj}\ \mathcal C_{\widetilde{Q}}\longrightarrow \mathcal{T}$$
 defined by the following rules:\\
(1)\ If $M$ is a $\mathfrak{S}$-module, then
                    $$
                       X_{M}=\sum_{\underline{e}} |\mathrm{Gr}_{\underline{e}} M|q^{-\frac{1}{2}
\langle
\underline{e},\underline{m}-\underline{e}\rangle}X^{-\widetilde{B}\underline{e}-(\widetilde{I}-\widetilde{R}^{'})\underline{m}};$$
(2)\ If $M$ is a $\mathfrak{S}$-module and $I$ is an injective
$\widetilde{\mathfrak{S}}$-module, then
                    $$
                       X_{M\oplus I[-1]}=\sum_{\underline{e}} |\mathrm{Gr}_{\underline{e}} M|q^{-\frac{1}{2}
\langle
\underline{e},\underline{m}-\underline{e}-\underline{i}\rangle}X^{-\widetilde{B}\underline{e}-(\widetilde{I}-\widetilde{R}^{'})\underline{m}+\underline{\mathrm{dim}}
\soc I},
                    $$
where $\underline{\mathrm{dim}} I= \underline{i},
\underline{\mathrm{dim}} M= \underline{m}$ and
$\mathrm{Gr}_{\underline{e}}M$ denotes the set of all submodules $V$
of $M$ with $\underline{\mathrm{dim}} V= \underline{e}$. We note
that
$$
X_{P[1]}=X_{\tau P}=X^{\underline{\mathrm{dim}} P/\rad
P}=X^{\underline{\mathrm{dim}}\soc I}=X_{I[-1]}=X_{\tau^{-1}I}.
$$
for any projective $\widetilde{\mathfrak{S}}$-module $P$ and
injective $\widetilde{\mathfrak{S}}$-module $I$ with $\soc I=P/\rad
P.$ In the following, we denote by the corresponding underlined
lower case
 letter $\underline{x}$ the dimension vector of a $\mathfrak{S}$-module
$X$ and view $\underline{x}$ as a column vector in $\mathbb{Z}^n.$

Now we need to recall some notations. For any
$\widetilde{\mathfrak{S}}$-modules $M,N$ and $E$, denote by
$\varepsilon_{MN}^{E}$  the cardinality of the set
$\mathrm{Ext}_{\widetilde{\mathfrak{S}}}^{1}(M,N)_{E}$ which is the
subset of $ \mathrm{Ext}_{\widetilde{\mathfrak{S}}}^{1}(M,N)$
consisting of those equivalence classes of short exact sequences
with middle term isomorphic to $E$ (\cite[Section 4]{Hubery}). Let
$F^M_{AB}$ be the number of submodules $U$ of $M$ such that $U$ is
isomorphic to $B$ and $M/U$ is isomorphic to $A$. Then by
definition, we have
$$|\mathrm{Gr}_{\underline{e}}(M)|=\sum_{A, B;
\underline{\mathrm{dim}}B=\underline{e}}F_{AB}^M.
$$
Denote by
$[M,N]^{1}=\mathrm{dim}_{k}\mathrm{Ext}_{\widetilde{\mathfrak{S}}}^{1}(M,N)$
and
$[M,N]=\mathrm{dim}_{k}\mathrm{Hom}_{\widetilde{\mathfrak{S}}}(M,N).$

Let $M,N$ be any $\mathfrak{S}-$modules  and $I$ any injective
$\widetilde{\mathfrak{S}}-$module. Define
$$\mathrm{Hom}_{\widetilde{\mathfrak{S}}}(M,I)_{BI'}:=\{f:M\longrightarrow I|\ker f\cong B, \cok f\cong
I'\}.$$ Note that $I'$ is an injective $\widetilde{\mfk{S}}-$module.

The main result of this article is the following theorem:
\begin{thm}\label{main}
$$(1)\ q^{[M,N]^{1}}X_{M}X_{N}=q^{\frac{1}{2}\Lambda((\widetilde{I}-\widetilde{R}^{'})\underline{m},
(\widetilde{I}-\widetilde{R}^{'})\underline{n})}
\sum_{E}\varepsilon_{MN}^{E}X_E,$$
$$(2)\ q^{[M,I]}X_{M}X_{I[-1]}=q^{\frac{1}{2}\Lambda((\widetilde{I}-\widetilde{R}^{'})\underline{m},
-\mathrm{\underline{dim}}\soc I)}
\sum_{B,I'}|\mathrm{Hom}_{\widetilde{\mfk{S}}}(M,I)_{BI'}|X_{B\oplus
I'[-1]}.$$
\end{thm}

\begin{definition}\label{def}
$X_{L}$ is called \emph{the corresponding cluster character}, if $L$
is a $\mfk{S}$-module or $L=M\oplus I[-1]\in\mathcal
C_{\widetilde{Q}}$ satisfying that $M$ is a $\mfk{S}$-module and $I$
is an injective $\widetilde{\mfk{S}}$-module.
\end{definition}
For a valued quiver $Q$, denote by $\mathcal{AH}_{|k|}(Q)$ the
 $\mathbb{ZP}$-subalgebra of $\mathcal{F}$ generated by
all the cluster characters and by $\mathcal{EH}_{|k|}(Q)$ the
corresponding quantum cluster algebra, i.e, the
$\mathbb{ZP}$-subalgebra of $\mathcal{F}$ generated by all the
cluster variables. Then we have the following corollary:
\begin{cor}\label{cor}
For any  valued quiver $Q$ of finite type or rank $2$, we have
$\mathcal{EH}_{|k|}(Q)=\mathcal{AH}_{|k|}(Q).$
\end{cor}

\noindent{\bbb 3\quad Proof of the main theorem}\\[0.1cm]
In this section, we fix a valued quiver $Q$  with $n$ vertices.

\begin{lem}\label{1} For any dimension vector
$\underline{m}, \underline{e}, \underline{f}\in \mathbb{Z}^{n}_{\geq
0},$ we have
$$(1)\ \Lambda((\widetilde{I}-\widetilde{R}^{'})\underline{m}, \widetilde{B}\underline{e})=-\langle \underline{e}, \underline{m}\rangle;$$
$$(2)\ \Lambda(\widetilde{B}\underline{e}, \widetilde{B}\underline{f})=\langle \underline{f}, \underline{e}\rangle-\langle \underline{e}, \underline{f}\rangle.$$
\end{lem}
\begin{proof}
By definition, we have \begin{eqnarray}
% \nonumber to remove numbering (before each equation)
   && \Lambda((\widetilde{I}-\widetilde{R}^{'})\underline{m}, \widetilde{B}\underline{e})  \nonumber\\
   &=& \underline{m}^{tr}(\widetilde{I}-\widetilde{R}^{'})^{tr}\Lambda \widetilde{B}\underline{e}=
   -\underline{m}^{tr}(\widetilde{I}-\widetilde{R}^{'})^{tr}\begin{bmatrix}D_n\\0 \end{bmatrix}\underline{e}\nonumber\\
   &=& -\underline{m}^{tr}(I_{n}-(R^{'})^{tr})D_n\underline{e}=-\underline{e}^{tr}D_n(I_{n}-R^{'})\underline{m}\nonumber\\
  &=& -\langle \underline{e}, \underline{m}\rangle.\nonumber
\end{eqnarray}
As for (2), the left side of the desired equation is equal to
$$\underline{e}^{tr}\widetilde{B}^{tr}\Lambda
\widetilde{B}\underline{f}=-\underline{e}^{tr}\widetilde{B}^{tr}\begin{bmatrix}D_n\\0
\end{bmatrix}\underline{f}=-\underline{e}^{tr}B^{tr}D_n\underline{f}.$$
The right side is
\begin{eqnarray}
% \nonumber to remove numbering (before each equation)
   && \langle \underline{f}, \underline{e}\rangle-\langle \underline{e}, \underline{f}\rangle  \nonumber\\
   &=& \underline{f}^{tr}D_n(I_{n}-R^{'})\underline{e}-\underline{e}^{tr}(I_{n}-R^{tr})D_n\underline{f}\nonumber\\
   &=& \underline{e}^{tr}(I_{n}-R^{'})^{tr}D_n\underline{f}-\underline{e}^{tr}(I_{n}-R^{tr})D_n\underline{f}\nonumber\\
  &=& \underline{e}^{tr}(R^{tr}-(R^{'})^{tr})D_n\underline{f}=-\underline{e}^{tr}B^{tr}D_n\underline{f}.\nonumber
\end{eqnarray}
Thus we prove the lemma.
\end{proof}
\begin{cor}\label{2}
For any dimension vector $\underline{m}, \underline{l},
\underline{e}, \underline{f}\in \mathbb{Z}^{n}_{\geq 0},$ we have
\begin{eqnarray}
% \nonumber to remove numbering (before each equation)
   && \Lambda(-\widetilde{B}\underline{e}-(\widetilde{I}-\widetilde{R}^{'})\underline{m},-
   \widetilde{B}\underline{f}-(\widetilde{I}-\widetilde{R}^{'})\underline{l})  \nonumber\\
  &=&\Lambda((\widetilde{I}-\widetilde{R}^{'})\underline{m},(\widetilde{I}-\widetilde{R}^{'})\underline{l})
  +\langle \underline{f}, \underline{e}\rangle-\langle \underline{e},
  \underline{f}\rangle+
  \langle \underline{e}, \underline{l}\rangle-\langle \underline{f}, \underline{m}\rangle.\nonumber
\end{eqnarray}
\end{cor}
\begin{proof}
It follows from Lemma \ref{1}.
\end{proof}\noindent\\[4mm]

\textit{Proof of  Theorem \ref{main}(1):} By Green's formula
\cite{Green}, we have
$$\sum_{E}\varepsilon_{MN}^{E}F^{E}_{XY}=\sum_{A,B,C,D}q^{[M,N]-[A,C]-[B,D]-\langle
A,D\rangle}F^{M}_{AB}F^{N}_{CD}\varepsilon_{AC}^{X}\varepsilon_{BD}^{Y}.$$
Then
\begin{eqnarray}
% \nonumber to remove numbering (before each equation)
   && \sum_{E}\varepsilon_{MN}^{E}X_E  \nonumber\\
   &=& \sum_{E,X,Y}\varepsilon_{MN}^{E}q^{-\frac{1}{2}\langle Y,X\rangle}F^{E}_{XY}X^{-\widetilde{B}\underline{y}-(\widetilde{I}-\widetilde{R}^{'})\underline{e}}\nonumber\\
  &=&\sum_{A,B,C,D,X,Y}q^{[M,N]-[A,C]-[B,D]-\langle
A,D\rangle-\frac{1}{2}\langle
B+D,A+C\rangle}F^{M}_{AB}F^{N}_{CD}\varepsilon_{AC}^{X}\varepsilon_{BD}^{Y}X^{-\widetilde{B}\underline{y}-(\widetilde{I}-\widetilde{R}^{'})\underline{e}}.\nonumber
\end{eqnarray}
By Corollary \ref{2}, we have
\begin{eqnarray}
% \nonumber to remove numbering (before each equation)
   && X^{-\widetilde{B}\underline{y}-(\widetilde{I}-\widetilde{R}^{'})\underline{e}} \nonumber\\
   &=& X^{-\widetilde{B}(\underline{b}+\underline{d})-(\widetilde{I}-\widetilde{R}^{'})(\underline{m}+\underline{n})}\nonumber\\
   &=& q^{\frac{1}{2}\Lambda(-\widetilde{B}\underline{d}-(\widetilde{I}-\widetilde{R}^{'})\underline{n},
-\widetilde{B}\underline{b}-(\widetilde{I}-\widetilde{R}^{'})\underline{m})}
X^{-\widetilde{B}\underline{b}-(\widetilde{I}-\widetilde{R}^{'})\underline{m}}X^{-\widetilde{B}\underline{d}-(\widetilde{I}-\widetilde{R}^{'})\underline{n}}\nonumber\\
   &=& q^{\frac{1}{2}\Lambda((\widetilde{I}-\widetilde{R}^{'})\underline{n},
(\widetilde{I}-\widetilde{R}^{'})\underline{m})+\frac{1}{2}[-\langle
D,B\rangle+\langle B,D\rangle+\langle D,M\rangle-\langle
B,N\rangle]}
X^{-\widetilde{B}\underline{b}-(\widetilde{I}-\widetilde{R}^{'})\underline{m}}X^{-\widetilde{B}\underline{d}-(\widetilde{I}-\widetilde{R}^{'})\underline{n}}\nonumber\\
  &=&q^{\frac{1}{2}\Lambda((\widetilde{I}-\widetilde{R}^{'})\underline{n},
(\widetilde{I}-\widetilde{R}^{'})\underline{m})}q^{\frac{1}{2}\langle
D,A\rangle-\frac{1}{2}\langle B,C\rangle}
X^{-\widetilde{B}\underline{b}-(\widetilde{I}-\widetilde{R}^{'})\underline{m}}X^{-\widetilde{B}\underline{d}-(\widetilde{I}-\widetilde{R}^{'})\underline{n}}.\nonumber
\end{eqnarray}
Thus
\begin{eqnarray}
% \nonumber to remove numbering (before each equation)
   && \sum_{E}\varepsilon_{MN}^{E}X_E  \nonumber\\
   &=& q^{-\frac{1}{2}\Lambda((\widetilde{I}-\widetilde{R}^{'})\underline{m},
(\widetilde{I}-\widetilde{R}^{'})\underline{n})}\sum_{A,B,C,D}q^{[M,N]-[A,C]-[B,D]-\langle
A,D\rangle-\frac{1}{2}\langle
B+D,A+C\rangle+[A,C]^{1}+[B,D]^{1}}\cdot\nonumber\\
  &&q^{\frac{1}{2}\langle
D,A\rangle-\frac{1}{2}\langle
B,C\rangle}F^{M}_{AB}F^{N}_{CD}X^{-\widetilde{B}\underline{b}-(\widetilde{I}-\widetilde{R}^{'})\underline{m}}X^{-\widetilde{B}\underline{d}-(\widetilde{I}-\widetilde{R}^{'})\underline{n}}
.\nonumber
\end{eqnarray}
Here we use the following fact
$$\sum_{X}\varepsilon_{AC}^{X}=q^{[A,C]^{1}},\sum_{Y}\varepsilon_{BD}^{Y}=q^{[B,D]^{1}}$$
An easy calculation shows that
$$[M,N]-[A,C]-[B,D]-\langle A,D\rangle+[A,C]^{1}+[B,D]^{1}=[M,N]^{1}+\langle B, C\rangle.$$
Hence
\begin{eqnarray}
% \nonumber to remove numbering (before each equation)
   && \sum_{E}\varepsilon_{MN}^{E}X_E  \nonumber\\
   &=& q^{-\frac{1}{2}\Lambda((\widetilde{I}-\widetilde{R}^{'})\underline{m},
(\widetilde{I}-\widetilde{R}^{'})\underline{n})}q^{[M,N]^{1}}\sum_{A,B,C,D}
F^{M}_{AB}q^{-\frac{1}{2}\langle
B,A\rangle}X^{-\widetilde{B}\underline{b}-(\widetilde{I}-\widetilde{R}^{'})\underline{m}}
F^{N}_{CD}q^{-\frac{1}{2}\langle
D,C\rangle}X^{-\widetilde{B}\underline{d}-(\widetilde{I}-\widetilde{R}^{'})\underline{n}}
\nonumber\\
 &=&q^{-\frac{1}{2}\Lambda((\widetilde{I}-\widetilde{R}^{'})\underline{m},
(\widetilde{I}-\widetilde{R}^{'})\underline{n})}q^{[M,N]^{1}}X_{M}X_{N}.\nonumber
\end{eqnarray}
This finishes the proof.\noindent\\[4mm]

\textit{Proof of  Theorem \ref{main}(2):} We calculate
\begin{eqnarray}
% \nonumber to remove numbering (before each equation)
   && X_{M}X_{I[-1]}  \nonumber\\
   &=& \sum_{G,H}q^{-\frac{1}{2}\langle H,G\rangle}F^{M}_{GH}X^{-\widetilde{B}\underline{h}-(\widetilde{I}-\widetilde{R}^{'})\underline{m}}
   X^{\mathrm{\underline{dim}}soc I}\nonumber\\
  &=& \sum_{G,H}q^{-\frac{1}{2}\langle H,G\rangle}F^{M}_{GH}q^{\frac{1}{2}\Lambda(-\widetilde{B}\underline{h}-(\widetilde{I}-\widetilde{R}^{'})\underline{m},
\mathrm{\underline{dim}}soc
I)}X^{-\widetilde{B}\underline{h}-(\widetilde{I}-\widetilde{R}^{'})\underline{m}+\mathrm{\underline{dim}}soc
I}
\nonumber\\
 &=&q^{\frac{1}{2}\Lambda(-(\widetilde{I}-\widetilde{R}^{'})\underline{m},
\mathrm{\underline{dim}}soc I)}\sum_{G,H}q^{-\frac{1}{2}\langle
H,G\rangle}q^{\frac{1}{2}\Lambda(-\widetilde{B}\underline{h},\mathrm{\underline{dim}}soc
I)}F^{M}_{GH}X^{-\widetilde{B}\underline{h}-(\widetilde{I}-\widetilde{R}^{'})\underline{m}+\mathrm{\underline{dim}}soc
I}\nonumber\\
 &=&q^{\frac{1}{2}\Lambda((\widetilde{I}-\widetilde{R}^{'})\underline{m},
-\mathrm{\underline{dim}}soc I)}\sum_{G,H}q^{-\frac{1}{2}\langle
H,G\rangle}q^{-\frac{1}{2}[H,I]}F^{M}_{GH}X^{-\widetilde{B}\underline{h}-(\widetilde{I}-\widetilde{R}^{'})\underline{m}+\mathrm{\underline{dim}}soc
I}.\nonumber
\end{eqnarray}
Here we use the fact that
$$\Lambda(-\widetilde{B}\underline{h},\mathrm{\underline{dim}}soc
I)=-\underline{h}^{tr}\widetilde{B}^{tr}\Lambda(\mathrm{\underline{dim}}soc
I)=-[H,I].$$
 Note that  we have the following commutative diagram
$$
\xymatrix{&0\ar[d]&0\ar[d]\\
&Y\ar@{=}[r]\ar[d]&Y\ar[d]\\
 0\ar[r]&B\ar[r]\ar[d]&M\ar[r]\ar[d]&A\ar@{=}[d]\ar[r]&0\\
0\ar[r]&X\ar[r]\ar[d]&G\ar[r]\ar[d]&A\ar[r]&0\\
&0&0}
$$
and short exact
 sequence
$$0\longrightarrow A\longrightarrow I\longrightarrow I'\longrightarrow 0,$$
Thus by \cite{Hubery} we have
$$\sum_{B}F^{B}_{XY}F^{M}_{AB}=\sum_{G}F^{G}_{AX}F^{M}_{GY},\
|\mathrm{Hom}_{\widetilde{\mfk{S}}}(M,I)_{BI'}|=\sum_{A}|\mathrm{Aut}(A)|F^{M}_{AB}F^{I}_{I'A}$$
and
$$\sum_{A,I',X}|\mathrm{Aut}(A)|F^{I}_{I'A}F^{G}_{AX}=\sum_{I',X}|\mathrm{Hom}_{\widetilde{\mfk{S}}}(G,I)_{XI'}|=q^{[
G,I]}=q^{\langle G,I\rangle}.$$ By \cite[Lemma 1]{Hubery}, we have
$(\widetilde{I}-\widetilde{R}^{'})\underline{i}=\mathrm{\underline{dim}}
soc I$. Now we can  calculate the  term
\begin{eqnarray}
% \nonumber to remove numbering (before each equation)
   &&  \sum_{B,I'}|\mathrm{Hom}_{\widetilde{\mfk{S}}}(M,I)_{BI'}|X_{B\oplus
I'[-1]}  \nonumber\\
   &=& \sum_{A,B,I',X,Y}|\mathrm{Aut}(A)|F^{M}_{AB}F^{I}_{I'A}q^{-\frac{1}{2}\langle
Y,X-I'\rangle}F^{B}_{XY}X^{-\widetilde{B}\underline{y}-(\widetilde{I}-\widetilde{R}^{'})\underline{b}+\mathrm{\underline{dim}}
soc I'}\nonumber\\
  &=& \sum_{A,G,I',X,Y}q^{-\frac{1}{2}\langle
Y,X-I'\rangle}|\mathrm{Aut}(A)|F^{I}_{I'A}F^{G}_{AX}F^{M}_{GY}X^{-\widetilde{B}\underline{y}-(\widetilde{I}-\widetilde{R}^{'})\underline{b}+\mathrm{\underline{dim}}
soc I'}.\nonumber
\end{eqnarray}
Note that  we have  the  following facts
$$\underline{i'}+\underline{a}=\underline{i},\
\underline{x}+\underline{a}=\underline{g}\Longrightarrow
\underline{x}-\underline{i'}=\underline{g}-\underline{i},$$ and
\begin{eqnarray}
% \nonumber to remove numbering (before each equation)
   &&  -\widetilde{B}\underline{y}-(\widetilde{I}-\widetilde{R}^{'})\underline{b}+\mathrm{\underline{dim}}
soc I' \nonumber\\
   &=&-\widetilde{B}\underline{h}-(\widetilde{I}-\widetilde{R}^{'})(\underline{m}-\underline{i}+\underline{i'})+\mathrm{\underline{dim}}
soc I'\nonumber\\
  &=& -\widetilde{B}\underline{h}-(\widetilde{I}-\widetilde{R}^{'})\underline{m}
  +(\widetilde{I}-\widetilde{R}^{'})(\underline{i}-\underline{i'})+\mathrm{\underline{dim}}
soc I'
\nonumber\\
 &=& -\widetilde{B}\underline{h}-(\widetilde{I}-\widetilde{R}^{'})\underline{m}
  +(\widetilde{I}-\widetilde{R}^{'})\underline{i}
\nonumber\\
 &=&-\widetilde{B}\underline{h}-(\widetilde{I}-\widetilde{R}^{'})\underline{m}
  +\mathrm{\underline{dim}}
soc I.\nonumber
\end{eqnarray}
Hence
\begin{eqnarray}
% \nonumber to remove numbering (before each equation)
   &&  \sum_{B,I'}|\mathrm{Hom}_{\widetilde{\mfk{S}}}(M,I)_{BI'}|X_{B\oplus
I'[-1]}  \nonumber\\
 &=&\sum_{G,H}q^{\langle
G,I\rangle}q^{-\frac{1}{2}\langle
H,G-I\rangle}F^{M}_{GH}X^{-\widetilde{B}\underline{h}-(\widetilde{I}-\widetilde{R}^{'})\underline{m}+\mathrm{\underline{dim}}
soc I}
\nonumber\\
 &=&\sum_{G,H}q^{\langle
M,I\rangle}q^{-\frac{1}{2}\langle H,I\rangle}q^{-\frac{1}{2}\langle
H,G\rangle}F^{M}_{GH}X^{-\widetilde{B}\underline{h}-(\widetilde{I}-\widetilde{R}^{'})\underline{m}+\mathrm{\underline{dim}}
soc I}\nonumber\\
 &=&q^{[
M,I]}\sum_{G,H}q^{-\frac{1}{2}[H,I]}q^{-\frac{1}{2}\langle
H,G\rangle}F^{M}_{GH}X^{-\widetilde{B}\underline{h}-(\widetilde{I}-\widetilde{R}^{'})\underline{m}+\mathrm{\underline{dim}}
soc I}.\nonumber
\end{eqnarray}
This finishes the proof.
\noindent\\[4mm]

  To prove Corollary \ref{cor}, we recall the following lemma which can be found in  \cite{CK2005}\cite{D}.
\begin{lem}\label{easy}
        Let $$M \longrightarrow E \longrightarrow N \xrightarrow{\epsilon} M[1]$$
        be a non-split triangle in $\mathcal C_{\widetilde{Q}}.$ Then
        $$\mathrm{dim}_{k}\mathrm{Ext}^{1}_{\mathcal C_{\widetilde{Q}}}(E,E) < \mathrm{dim}_{k}\mathrm{Ext}^{1}_{\mathcal C_{\widetilde{Q}}}(M \oplus N, M \oplus N).$$
\end{lem}

\textit{Proof of  Corollary \ref{cor}:} Firstly, we prove that for
any indecomposable object $M\in\mathcal C_{\widetilde{Q}}$, $X_M$ is
in the quantum cluster algebra $\mathcal{EH}_{|k|}(Q)$.

When $Q$ is a valued quiver of finite type, it follows that $X_M$ is
a cluster variable for any indecomposable object $M\in\mathcal
C_{\widetilde{Q}}$ by \cite{rupel}.

When $Q$ is a valued quiver of rank $2$. Denoted by
$$\Phi_{i}: \mathcal{A}_{|k|}(Q)\rightarrow
\mathcal{A}_{|k|}(Q')$$ the canonical isomorphism of quantum cluster
algebras associated to sink or source $1\leq i\leq 2$. Let
$\Sigma_i:\ \mathrm{mod}Q \longrightarrow \ \mathrm{mod}Q'$ be the
standard BGP-reflection functor. It follows from \cite[Theorem
2.4]{rupel} that $\Phi_{i}(X_M^{Q})=X_{\Sigma_iM}^{Q'}$ for any
regular module $M$ of $Q$. Note also  the fact $Q$ is an acyclic
valued quiver of rank $2$, we have that $X_M$ is in the upper
quantum cluster algebra associated to $Q$ which coincides with the
quantum cluster algebra $\mathcal{EH}_{|k|}(Q)$ by the acyclicity of
$Q$ \cite{berzel}. When  $M$ is an indecomposable preprojective or
preinjective module, it follows from \cite{rupel} $X_M$ is a cluster
variable, hence in the quantum cluster algebra
$\mathcal{EH}_{|k|}(Q)$.

Now we  need to prove that for any cluster character
$X_{L}\in\mathcal{AH}_{|k|}(Q)$, then
$X_{L}\in\mathcal{EH}_{|k|}(Q)$. Let $L\cong
\bigoplus_{i=1}^{l}L_{i}^{\oplus n_{i}}, n_{i}\in \mathbb{N}$ where
$L_{i}\ (1\leq i\leq l)$ are indecomposable objects in $\mathcal
C_{\widetilde{Q}}$. By Theorem \ref{main} and Lemma \ref{easy}, we
have that
$$X^{n_{1}}_{L_{1}}X^{n_{2}}_{L_{2}}\cdots X^{n_{l}}_{L_{l}}=q^{\frac{1}{2}n_{L}}X_{L}+
\sum_{\dim_{k}\mathrm{Ext}^{1}_{\mathcal
C_{\widetilde{Q}}}(E,E)<\dim_{k}\mathrm{Ext}^{1}_{\mathcal
C_{\widetilde{Q}}}(L,L)}f_{n_{E}}(q^{\pm\frac{1}{2}})X_E$$ where
$n_{L}\in \mathbb{Z}$ and $f_{n_{E}}(q^{\pm\frac{1}{2}})\in
\mathbb{Z}[q^{\pm\frac{1}{2}}].$ Note that the left side of the
equation above is in  $\mathcal{EH}_{|k|}(Q)$, thus by induction, it
follows  that $X_{L}\in \mathcal{EH}_{|k|}(Q)$ which finishes the
proof.
\noindent\\[4mm]

\noindent{\bbb 3\quad Bases in the quantum cluster algebras of rank $2$}\\[0.1cm]
In this section, we consider   a valued quiver (see \cite{rupel} for
details) associated to a given compatible pair $(\Lambda,
B)$ where  $\Lambda=\left(\begin{array}{cc} 0 & 1\\
-1 &
0\end{array}\right)$ and $B=\left(\begin{array}{cc} 0 & b\\
-c & 0\end{array}\right)$ for any $b,c\in \mathbb{Z}_{> 0}$. Let
$\mathcal{T}=\ZZ[q^{\pm 1/2}]\langle X_1^{\pm1}, X_2^{\pm1}:
X_1X_2=qX_2X_1\rangle$ and ${\mathcal F}$ be the skew field of
fractions of $\mathcal{T}$ and thus the quantum cluster algebra of
the valued quiver of rank $2$ (denoted by $\mathcal{A}_q(b,c)$ in
the sequel) is the $\ZZ[q^{\pm 1/2}]$-subalgebra of ${\mathcal F}$
generated by the cluster variables $X_k$, $k\in\ZZ$, defined
recursively by
\[X_{m-1}X_{m+1}=\begin{cases}
q^{\frac{b}{2}}X^{b}_{m}+1& \text{if $m$ is odd;}\\
q^{\frac{c}{2}}X^{c}_{m}+1 & \text{if $m$ is even.}
\end{cases}
\]

\begin{definition}\label{defp}
For  any $(r_{1},r_{2})$ and $(s_{1},s_{2})\in \mathbb{Z}^{2}$, we
write $(r_{1},r_{2})\preceq (s_{1},s_{2})$ if $r_{i}\leq s_{i}$ for
$1\leq i\leq 2$. Moreover, if there exists some $i$ such that
$r_{i}< s_{i}$, then we write $(r_{1},r_{2})\prec (s_{1},s_{2}).$
\end{definition}
For any $\underline{m}\in \mathbb{Z}^{2},$ define
$\underline{m}^{+}=(m^+_1,m^+_2)$ such that $m^+_i=m_i$ if $m_i>0$
and $m_i^+=0$ if $m_i\leq 0$ for any $1\leq i \leq 2.$ Dually, we
set $\underline{m}^-=\underline{m}^+-\underline{m}.$  Denote by
$\underline{dim}I[-1]=-\underline{dim}soc I$ for any injective
module $I$. For any $\underline{d}\in \mathbb{Z}^{2},$ we make the
following assignment by
$$X_{\underline{d}}:=X_{M}\ \ \ \text{for any}\ M\in\mathcal C_{Q}\ \text{with}\ \underline{dim}M=\underline{d}.$$
Note that this assignment is not unique.

\begin{thm}
The set
$\mathcal{B}=\{X_{\underline{d}}|\underline{d}\in\mathbb{Z}^{2}\}$
is a $\mathbb{Z}[q^{\pm \frac{1}{2}}]-$basis of the quantum cluster
algebra $\mathcal{A}_q(b,c)$.
\end{thm}
\begin{proof}
 Note that for any $\underline{d}\in \mathbb{Z}^2$,  $X_{\underline{d}}\in \mathcal{A}_q(b,c)$ by Corollary \ref{cor}.
According to the definition of the quantum Caldero-Chapoton map and
the partial order in Definition \ref{defp}, we obtain  a minimal
term $a_{\underline{d}}X^{\underline{d}}$ in the laurent expansion
in $X_{\underline{d}}$, for some nonzero
$a_{\underline{d}}\in\mathbb{Z}[q^{\pm \frac{1}{2}}]$. Then by  the
standard monomials in \cite{berzel}, we have
$$X_{\underline{d}}=b_{\underline{d}}X^{d^-_1}_{1}X^{d^-_2}_{2}X^{d^+_1}_{S_1}X^{d^+_2}_{S_2}+\sum_{\underline{d}\succ
\underline{l}}b_{\underline{l}}X^{l^-_1}_{1}X^{l^-_2}_{2}X^{l^+_1}_{S_1}X^{l^+_2}_{S_2}$$
where  $b_{\underline{d}},b_{\underline{l}}\in \mathbb{Z}[q^{\pm
\frac{1}{2}}]$. It is easy to see that $b_{\underline{d}}$ must be
some nonzero monomial in $q^{\pm\frac{1}{2}}$. Thus we obtain that
$\mathcal{B}$ is a $\mathbb{Z}[q^{\pm \frac{1}{2}}]$-basis of
$\mathcal{A}_q(b,c)$.
\end{proof}

 \noindent{\bbb{References}}

\end{document}